\documentclass[12pt,dvips,preprint]{imsart}
\usepackage{amsthm,amsmath,amssymb,bm}
\usepackage{graphicx}
\newtheorem{theorem}{Theorem}
\newtheorem{col}{Corollary}
\newtheorem{lemma}{Lemma}
\newtheorem{prop}{Proposition} 

\startlocaldefs
\endlocaldefs
\allowdisplaybreaks
\begin{document}

\begin{frontmatter}

\title{Strong Consistency of Reduced $K$-means Clustering}
\runtitle{Consistency of RKM Clustering}


\author{\fnms{Yoshikazu} \snm{Terada}\corref{}\ead[label=e1]{terada@sigmath.es.osaka-u.ac.jp}}
\address{Graduate School of Engineering Science, Osaka University, 1-3 Machikaneyama,\\ Toyonaka, Osaka, Japan\\ \printead{e1}}
\affiliation{Osaka University}

\runauthor{Y. Terada}

\begin{abstract}
Reduced $k$-means clustering is a method for clustering objects in 
a low-dimensional subspace.
The advantage of this method is that
both clustering of objects and low-dimensional subspace reflecting 
the cluster structure are simultaneously obtained.
In this paper,
the relationship between conventional $k$-means clustering 
and reduced $k$-means clustering is discussed.
Conditions ensuring almost sure convergence of 
the estimator of reduced $k$-means clustering as unboundedly increasing sample size have been presented.
The results for a more general model considering conventional $k$-means clustering and reduced $k$-means clustering
are provided in this paper.
Moreover, 
a new criterion and its consistent estimator are proposed to determine the optimal dimension number of a subspace, given the number of clusters.
\end{abstract}



\end{frontmatter}
\section{Introduction}\label{section:1}
The aim of cluster analysis 
is 
the discovery of a finite number of homogeneous classes from data.
In some cases, 
a cluster structure is considered to lie in a low-dimensional subspace of data,
and the following procedure is applied:
\begin{description}
\item[Step $1$.]
Principal component analysis (PCA) is performed, and the first few components are obtained. 
\item[Step $2$.] 
Conventional $k$-means clustering is performed for the principal scores on the first few principal components.
\end{description}
This two-step procedure is called ``tandem clustering" by Arabie \& Hubert (1994)
and 
has been discouraged by several authors (e.g., Arabie \& Hubert, 1994; Chang, 1983; De Soete \& Carroll, 1994).
Because the first few principal components of PCA do not
necessarily reflect the cluster structure in data, 
the appropriate clustering result may not be obtained by using the tandem clustering approach.
Figure $\ref{PCA:biplot}$ shows that 
the first two principal components do not reflect the cluster structure, and 
the clustering result of the tandem clustering is incorrect. 
\begin{figure}
\includegraphics[scale=0.35]{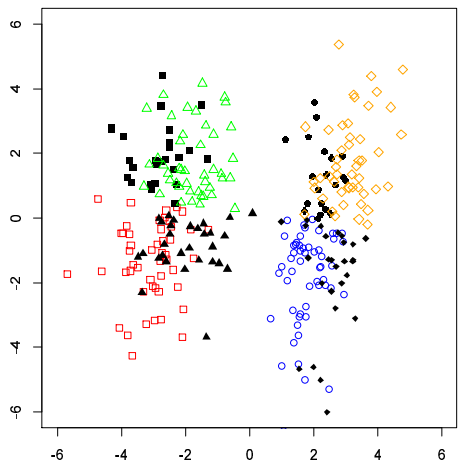}
\caption{First two dimensions of the principal component analysis and result of the tandem clustering (Black points represent the misclassification objects).}
\label{PCA:biplot}
\end{figure}
De Soete \& Carroll (1994) proposed reduced $k$-means (RKM) clustering.
RKM clustering simultaneously determines the clusters of objects on the basis of the $k$-means criterion 
and 
the subspace that is informative about the cluster structure in data on the basis of component analysis.
In other words, 
for given data points $\bm{x}_1,\;\dots,\;\bm{x}_n$ in $\mathbb{R}^{p}$, 
the fixed cluster number $k$ and the dimension number of subspace $q\;(q<\min\{k-1,\;p\})$, RKM clustering is defined by the minimization problem of the following loss function:
\begin{align}\label{eq:1.1}
RKM_n:=\frac{1}{n}\sum_{i=1}^{n}\min_{1\le j \le k} \|\bm{x}_i-A\bm{f}_j \|^2,
\end{align}
where $\bm{f}_j\in \mathbb{R}^{q}$ and $A$ is a $p\times q$ columnwise orthonormal matrix.
For some clustering methods related to $k$-means clustering, 
several authors have discussed their statistical properties 
(e.g., Abraham et al., 2003; Garc\'{i}a-Escudero et al., 1999; Pollard, 1981; Pollard, 1982; von Luxburg et al., 2008). 
However, 
because RKM clustering is proposed in the framework of descriptive statistics, 
the statistical properties are not discussed.
When data points are independently drawn from a population distribution $P$, 
the objective function is rewritten as 
$$
RKM(F,A,P_n):=\int\min_{\bm{f}\in F}\|\bm{x}-A\bm{f}\|P_n(d\bm{x}),
$$
where $F$ is a set containing $k$ or fewer points in $\mathbb{R}^{q}$, and $P_n$ is the empirical measure obtained from the data.
For each fixed $F$ and $A$, 
the strong law of large numbers (SLLN) shows that
$$
\lim_{n\rightarrow \infty}RKM(F,\;A,\;P_n)= RKM(F,\;A,\;P):=\int\min_{\bm{f}\in F}\|\bm{x}-A\bm{f}\|P(d\bm{x})\quad\mathrm{a.s.}
$$
Thus, we wish to ensure that the global minimizer of $RKM(\cdot,\cdot,\;P_n)$ converges almost surely to 
the global minimizers of $RKM(\cdot,\;\cdot,\;P)$, say the population global minimizers.

In this paper, 
the strong consistency of RKM under i.i.d. sampling is proven.
For this purpose, 
the framework of the proof of the strong consistency of the $k$-means clustering approach proposed by Pollard (1981) is used;
in this framework,  
the existence and uniqueness of the population global minimizers are assumed for consistency.
Conditions for the existence of the global minimizers are not discussed.
For RKM clustering, the uniqueness of the population global minimizers cannot be assumed
because RKM clustering has rotational indeterminacy.
Therefore,
the sufficient condition for the existence of the population global minimizers must be derived; 
it is also necessary to establish that the distance between the sample estimator and the {\it set} of global minimizers converges almost surely to zero, as the sample size approaches infinity.  

This paper is organized as follows.
In Section $\ref{section:2}$, 
the original algorithm of RKM clustering and visualization of the result are described.
Then, the relationship between the conventional $k$-means clustering method and RKM clustering is presented.
The notation and some properties of RKM, including the rotational indeterminacy, is introduced in Section $\ref{section:3}$. 
The uniform SLLN and continuity of the objective function of RKM clustering
are presented in Section $\ref{section:4}$.
In Section $\ref{section:5}$, 
conditions for the existence of the population global minimizers are determined, 
and  
a theorem regarding the strong consistency of RKM clustering is stated.
In Section $\ref{section:6}$, 
the main proof of the consistency theorem is explained.
In Section $\ref{section:7}$, 
a new criterion and its consistent estimator are proposed to determine the optimal dimension number of a subspace, given the number of clusters.
Moreover,
the effectiveness of the criterion through numerical experiments are illustrated. 
\section{Reduced $k$-means clustering}\label{section:2}
\subsection{Algorithm and visualization of reduced $k$-means clustering}

Let $X=(x_{ij})_{n\times p}$ be a data matrix and 
$\bm{x}_{i}\;(i=1,\;\dots,\;n)$ be row vectors of $X$,
where $n$ is the number of objects and $p$ is the number of variables. 
The number of clusters and components to which the variables are reduced are denoted by $k$ and $q$, respectively. 
RKM clustering is defined as the minimizing problem of the following criterion:
\begin{align}\label{eq:2.1}
RKM_n(A,\;F,\;U\mid k,\;q):=\|X-UFA^{T}\|_F^2 =\sum_{i=1}^{n}\min_{1\le j\le k}\|\bm{x}_i-A\bm{f}_j\|^2,
\end{align}
where $\|\cdot\|$ and $\|\cdot\|_F$ denote the usual Euclidean norm and Frobenius norm, respectively,
$U=(u_{ij})_{n\times k}$ is a binary membership matrix that
specifies cluster membership for each objects, 
$A=(a_{ij})_{p\times q}$ is a column-wise orthonormal loading matrix, 
$F=(f_{ij})_{k\times q}$ is a centroid matrix, 
and $\bm{f}_{j}$ is a centroid of the $j$th cluster for each $j=1,\;\dots,\;k$.
For example, this problem can be solved by the following alternating least square algorithm:
\begin{description}
\item[Step $0$.] First, initial values are chosen for $A,\;F,$ and $U$.
\item[Step $1$.] 
$Q\Sigma P^T$ is expressed as the singular value decomposition of $(UF)^TX$, 
where $Q$ is a $q\times q$ orthonormal matrix, 
$\Sigma$ is a $q\times q$ diagonal matrix, and $P$ is a $p \times q$ columnwise orthonormal matrix.
$A$ is updated by $PQ^T$.
\item[Step $2$.] 
For each $i=1,\;\dots,\;n$ and each $j=1,\;\dots,\;k$, 
we update $u_{ij}$ by 
\begin{align*}
u_{ij} = 
\begin{cases}
1 & \text{iff $\|A^T\bm{x}_i-\bm{f}_j\|^2 < \|A^T\bm{x}_i-\bm{f}_{j^{\prime}}\|^2$ for each $j^{\prime}\neq j$},\\
0 & \text{otherwise}.
\end{cases}
\end{align*}
\item[Step $3$.] $F$ is updated using $(U^{T}U)^{-1}U^{T}XA$.
\item[Step $4$.] 
Finally, the value of the function $RKM_n$ for the present values of $A,\;F$, and $U$ is computed.
When the present values have decreased the function value, 
$A,\;F$, and $U$ are update in accordance with Steps $1$--$3$.
Otherwise, the algorithm has converged.
\end{description}
Other formulations and algorithms for RKM clustering have been presented by De Soete \& Carrol (1994) and Timmerman et al. (2010).

The algorithms for RKM clustering monotonically decrease the function $RKM_n$.
As shown below, because $RKM_n$ is bounded, the solution for each iteration converges to a local minimum point.
Because of the binary constraint on $U$, 
the solutions of these algorithms may often be local minimums.  
To prevent this, many random starts are required to be used.

The objective function $RKM_n$ can be decomposed into two terms:
\begin{align}\label{eq:2.2}
RKM_n(A,\;F,\;U\mid k,\;q) = \|X-XAA^T\|_{F}^2 + \| XA- UF \|_{F}^2.
\end{align}
The first term of equation $(\ref{eq:2.2})$ is the objective function of the PCA, 
and 
the second term is the $k$-means criterion in a low dimensional subspace.
Thus, for optimal solutions $\hat{A},\;\hat{F}$, and $\hat{U}$, we have $\hat{F}=(\hat{U}^T\hat{U})^{-1}\hat{U}^T\hat{X}\hat{A}$. 
Using the optimal solutions $\hat{A}$, $\hat{F}$, and $\hat{U}$, 
the low-dimensional representation of the objects and cluster centers can be obtained:
\begin{eqnarray}
Y:=X\hat{A}\quad \text{ and } \quad G:=(\hat{U}^T\hat{U})^{-1}\hat{U}^TY.
\end{eqnarray}
Using $Y$ and $\hat{A}$,  
a biplot reflecting the cluster structure can be presented.
Figure $\ref{RKM:biplot}$ shows the biplot of the RKM clustering for the same data as that used in Figure $\ref{PCA:biplot}$.
\begin{figure}
\includegraphics[scale=0.35]{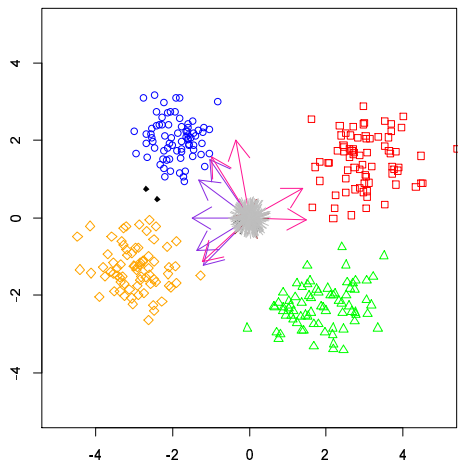}
\caption{Biplot of the result of RKM clustering for the same data set as Figure $1$ (Black points represent the misclassification objects).}
\label{RKM:biplot}
\end{figure}
\subsection{The relationship between the conventional $k$-means and the RKM clusterings}

The objective function of the conventional $k$-means clustering method is given by
\begin{align}\label{eq:2.3}
KM_n(C,\;U\mid k):=\|X-UC\|_F^2,
\end{align}
where $C$ is an $k\times p$ cluster center matrix.
$P\Sigma Q^{T}$ is expressed as the singular value decomposition of $C$, 
where $P$ is an $k\times k$ orthonormal matrix, 
$\Sigma$ is an $k\times k$ diagonal matrix, 
and $Q$ is a $p\times k$ column-wise orthonormal matrix.
Function $(\ref{eq:2.3})$ can be expressed as 
\begin{align*}
\|X-UC\|^2&=\|X -UP\Sigma Q^{T}\|_F^2.
\end{align*}
Considering $P\Sigma$ and $Q$ as a low-dimensional centroid matrix $F$ and a loading matrix $A$, respectively, 
function $(\ref{eq:2.3})$ is equivalent to the objective function of RKM, $RKM_n(A,\;F,\;U\mid k,\;k)$. 
Thus, RKM clustering includes the conventional $k$-means clustering analysis as a special case.

\section{Preliminaries}\label{section:3}
Let $(\Omega,\;\mathcal{F},\;P)$ be a probability space
and $\bm{X}_1,\;\dots,\;\mathbf{X}_n$ be independent random variables with a common population distribution $P$ on $\mathbb{R}^{p}$;
let $P_n$ be the empirical measure based on $\bm{X}_1,\;\dots,\;\bm{X}_n$.
For typographical convenience, 
the set of all $p\times q$ column-wise orthonormal matrices are denoted by $\mathcal{O}(p\times q)$, 
and $\mathcal{R}_k:=\{ R \subset \mathbb{R}^{q} \mid \#(R)\le k \}$, where $\#(R)$ is the cardinality of $R$.
Thus, the parameter space is denoted by $\Xi_k := \mathcal{R}_k \times \mathcal{O}(p\times q)$.
$B_q(r)$ denotes the $q$-dimensional closed ball of radius $r$ centered at the origin.
For each $M>0$, define $\mathcal{R}_k^{\ast}(M):= \{R \subset B_q(M) \mid \#(R) \le k\}$ 
and $\Theta_k^{\ast}(M):=\mathcal{R}_k^{\ast}(M) \times \mathcal{O}(p\times q)$.
Let $\phi:\mathbb{R}\rightarrow \mathbb{R}$ be a non-negative decreasing function
and $Q$ be a probability measure on $\mathbb{R}^{p}$.
For
each finite subset $F\subset \mathbb{R}^{q}$ and
each $A\in\mathcal{O}(p\times q)$,
the loss function of RKM with $Q$ is defined by
$$
\Phi(F,\;A,\;Q) := \int \min_{\bm{f}\in F}\phi(\|\bm{x}-A\bm{f}\|)Q(d\bm{x}).
$$
Write
$$
m_{k}(Q):=\inf_{(F,\;A)\in \Xi_k}\Phi(F,\;A,\;Q) 
\quad\text{and}\quad 
m_{k}^{\ast}(Q\mid M):=\inf_{(F,\;A)\in \Theta_k^{\ast}(M)}\Phi(F,\;A,\;Q).
$$
For $\theta=(F,\;A)\in \Xi_k$, 
both descriptions $\Phi(\theta,\;Q)$ and $\Phi(F,\;A,\;Q)$ are used.
In addition, 
$\Theta^{\prime}:= \{ \theta\in \Xi_k\mid m_k(P)=\Phi(\theta,\;P) \}$
and
$\Theta_n^{\prime}:= \{ \theta \in \Xi_k\mid m_k(P_n)=\Phi(\theta,\;P_n) \}$.
For each $M>0$, 
$\Theta^{\ast}:= \{ \theta\in \Theta_k^{\ast}(M)\mid m_k^{\ast}(P_n\mid M)=\Phi(\theta,\;P_n) \}$
and 
$\Theta_n^{\ast}:= \{ \theta \in \Theta_k^{\ast}(M)\mid m_k^\ast(P_n\mid M)=\Phi(\theta,\;P_n) \}$.
The parameters $\Theta^{\prime}(k)$ and $\Theta_n^{\prime}(k)$ are used to emphasize that
$\Theta^{\prime}$ and $\Theta_n^{\prime}$ are dependent on the index $k$. 
One of the measurable estimators in $\Theta_n^{\prime}$ will be denoted by $\hat{\theta}_n$ or $\hat{\theta}_n(k)$.
Similarly, 
we will also denote one of the measurable estimators in $\Theta_n^{\ast}$ 
by $\hat{\theta}_n^{\ast}$ or $\hat{\theta}_n^{\ast}(k)$.
To illustrate the existence of measurable estimators, see Section 6.7 of Pfanzagl (1996). 

Let $d_F(\cdot,\;\cdot)$ be the distance between two matrices based on Frobenius norm and $d_H(\cdot,\;\cdot)$ the Hausdorff distance, which is defined for finite subsets $A,\;B\subset \mathbb{R}^{q}$ as
$$
d_H(A,\;B):=\max_{\bm{a}\in A}\left\{\min_{\bm{b}\in B}\|\bm{a}-\bm{b}\| \right\}.
$$
Moreover, let $d$ be the product distance with $d_F$ and $d_H$.
In this paper, the distance between $\hat{\theta}_n$ and $\Theta^{\prime}$ is defined as 
$$
d(\hat{\theta}_n,\;\Theta^{\prime}):=\inf\{d(\hat{\theta}_n,\;\theta)\mid \theta \in \Theta^{\prime}\}.
$$

To clarify the minimization procedures, 
the function $\phi$ must satisfy some regularity conditions.
As proposed by Pollard (1981), 
it is assumed that $\phi$ is continuous, and $\phi(0)=0$.
Moreover, 
to control the growth of $\phi$,
it is assumed that
$$
\exists \lambda >0;\;\forall r >0;\;\phi(2r)\le \lambda\phi(r).
$$

For each $\bm{f}\in \mathbb{R}^{q}$ and each $A\in \mathcal{O}(p\times q)$,
\begin{align*}
\int\phi(\|\bm{x}-A\bm{f}\|)P(d\bm{x})
&\le 
\int\phi(\|\bm{x}\|+\|A\bm{f}\|)P(d\bm{x})
=
\int\phi(\|\bm{x}\|+\|\bm{f}\|)P(d\bm{x})\\
&=
\int_{\|\bm{f}\|>\|\bm{x}\|}\phi(2\|\bm{f}\|)P(d\bm{x}) + 
\int_{\|\bm{f}\|\le\|\bm{x}\|}\phi(2\|\bm{x}\|)P(d\bm{x})\\
&\le 
\phi(2\|\bm{f}\|) 
+ 
\lambda\int\phi(\|\bm{x}\|)P(d\bm{x}).
\end{align*}
Therefore,
as long as $\int \phi(\|\bm{x}\|)P(d\bm{x})$ is finite,
$\Phi(F,\;A,\;P)$ is also finite for each $F$ and each $A \in \mathcal{O}(p\times q)$.

Let $R$ be a $q \times q$ orthonormal matrix, i.e., $R^{T}R=RR^T=I_q$.
For each $\bm{f}\in \mathbb{R}^{q}$ and each $A \in \mathcal{O}(p\times q)$, 
\begin{align*}
\int\phi(\|\bm{x}-A\bm{f}\|)P(d\bm{x}) = \int\phi(\|\bm{x}-AR^{T}R\bm{f}\|)P(d\bm{x}).
\end{align*}
It follows that $\Theta^{\prime}$ is not a singleton when $\Theta^{\prime}\neq \emptyset$,
thus suggesting that RKM clustering has rotational indeterminacy.
\section{The uniform SLLN and the continuity of $\Phi(\cdot,\;\cdot,\;P)$}\label{section:4}
\begin{prop}\label{prop:USLLN}
Let $M>0$ be an arbitrary number. 
Let $\mathcal{G}$ denote the class of all $P$-integrable functions on $\mathbb{R}^p$ of the form 
$$
g_{(F,\;A)}(\bm{x}):=\min_{\bm{f}\in F}\phi(\|\bm{x}-A\bm{f}\|),
$$
where $(F,\;A)$ takes all values over $\Theta_{k}^{\ast}(M)$.
Suppose that $\int\phi(\|\bm{x}\|)P(d\bm{x})<\infty$.
Then, 
\begin{align}\label{eq:4.1}
\lim_{n\rightarrow \infty}\sup_{g\in \mathcal{G}}\left| \int g(\bm{x})\,P_n(d\bm{x}) -\int g(\bm{x})\,P(d\bm{x}) \right| =0 \quad \mathrm{a.s.} 
\end{align}
\end{prop}
\begin{proof}
DeHardt (1971) provided the sufficient condition for the uniform SLLN $(\ref{eq:4.1})$;
for all $\epsilon>0$, there exists a finite class of functions $\mathcal{G}_{\epsilon}$ 
such that 
for each $g\in \mathcal{G}$, $\dot{g}$ and $\bar{g}$ exist in $\mathcal{G}_\epsilon$ with $\dot{g} \le g \le \bar{g}$
and
$\int \bar{g}(\bm{x})\,P(d\bm{x})-\int \dot{g}(\bm{x})\,P(d\bm{x}) <\epsilon$.

An arbitrary $\epsilon>0$ is selected, 
and
$S_{p\times q}(\sqrt{q})$ denotes the surface of the sphere on $\mathbb{R}^{p\times q}$ of radius $\sqrt{q}$ centered at the origin.
To find such a finite class $\mathcal{G}_{\epsilon}$, 
$D_{\delta_1}$ is defined as the finite set of $\mathbb{R}^{q}$ satisfying
$$
\forall \bm{f} \in B_q(M);\;\exists \bm{g}\in D_{\delta_1};\;\|\bm{f}-\bm{g}\|<\delta_1
$$
and
$\mathcal{A}_{p\times q,\;\delta_2}$ as the finite sets of $S_{p\times q}(\sqrt{q})$ satisfying
$$
\forall A \in S_{p\times q}(\sqrt{q});\; \exists B \in \mathcal{A}_{p\times q,\;\delta_2};\;\| A- B \|_F <\delta_2.
$$
Define $\mathcal{R}_{k,\;\delta_1}:=\{F \in \mathcal{R}_{k}^\ast(M)\mid F \subset D_{\delta_1}\}$.
Take $\mathcal{G}_\epsilon$ as the finite class of functions of the form
$$
\min_{\bm{f}\in F^{\prime}}\phi(\|\bm{x}-A^{\prime}\bm{f}\|+\sqrt{q}\delta_1+M\delta_2) \quad \text{or} \quad 
\min_{\bm{f}\in F^{\prime}}\phi(\|\bm{x}-A^{\prime}\bm{f}\|-\sqrt{q}\delta_1-M\delta_2),
$$
where $(F^{\prime},\;A^{\prime})$ takes all values over $\mathcal{R}_{k,\;\delta_1}\times \mathcal{A}_{p\times q,\;\delta_2}$ 
and $\phi(r)$ is defined as zero for all negative $r<0$.

For given $F=\{\bm{f}_1,\;\dots,\;\bm{f}_k\}\in \mathcal{R}_{k}^\ast(M)$ 
and $A\in \mathcal{O}(p\times q)$,
there exists $F^{\prime}=\{\bm{f}_1^{\prime},\;\dots,\;\bm{f}_k^{\prime}\}\in \mathcal{R}_{k,\;\delta_1}$ 
with $\|\bm{f}_i-\bm{f}_i^{\prime}\|<\delta_1$ for each $i$ 
and each $A^{\prime}\in \mathcal{A}_{p\times q,\;\delta_2}$ with
$\|A-A^{\prime}\|_F<\delta_2$.
Corresponding to each $g_{(F,\;A)}\in \mathcal{G}$,
choose
$$
\bar{g}_{(F,\;A)}:= 
\min_{\bm{f}\in F^{\prime}}\phi(\|\bm{x}-A^{\prime}\bm{f}\|+\sqrt{q}\delta_1+M\delta_2)
$$
and
$$
\dot{g}_{(F,\;A)}:= 
\min_{\bm{f}\in F^{\prime}}\phi(\|\bm{x}-A^{\prime}\bm{f}\|-\sqrt{q}\delta_1-M\delta_2).
$$
Because $\phi$ is a monotone function and
\begin{align*}
\|\bm{x}-A^{\prime}\bm{f}_{i}^{\prime}\|-\sqrt{q}\delta_1-M\delta_2
\le
\|\bm{x}-A\bm{f}_{i}\|\le \|\bm{x}-A^{\prime}\bm{f}_{i}^{\prime}\|+\sqrt{q}\delta_1+M\delta_2
\end{align*}
for each $i$ and each $\bm{x}\in \mathbb{R}^{p}$,
these functions ensure that $\dot{g}_{(F,\;A)} \le g_{(F,\;A)} \le \bar{g}_{(F,\;A)}$.

If we choose $R>0$ to be greater than $\sqrt{q}\delta_1+M\delta_2+M\sqrt{q}$,
\begin{align*}
&\int \left[ \bar{g}_{(F,\;A)}(\bm{x}) - \dot{g}_{(F,\;A)}(\bm{x})\right] P(d\bm{x})\\
\le
&\int \sum_{i=1}^{k}\biggl[
\phi(\|\bm{x}-A^{\prime}\bm{f}_{i}^{\prime}\|+\sqrt{q}\delta_1+M\delta_2)\\
&\qquad\qquad-
\phi(\|\bm{x}-A^{\prime}\bm{f}_{i}^{\prime}\|-\sqrt{q}\delta_1-M\delta_2) 
\biggr]P(d\bm{x})\\
\le
&k\sup_{\|\bm{x}\|\le R}\sup_{\bm{f}\in B(5M)}\sup_{A\in S_{p\times q}(\sqrt{q})}
\biggl[
\phi(\|\bm{x}-A\bm{f}\|+\sqrt{q}\delta_1+M\delta_2)\\
&\qquad\qquad-
\phi(\|\bm{x}-A\bm{f}\|-\sqrt{q}\delta_1-M\delta_2)
\biggr]
+ 2k\lambda \int_{\|\bm{x}\|\ge R} \phi(\|\bm{x}\|)P(d\bm{x}).
\end{align*}
The second term would be less than $\epsilon/2$ if $R$ is sufficiently large.
Moreover, 
because $\phi$ is uniform continuous on a bounded set, 
the first term can be less than $\epsilon/2$ if $\delta_1,\;\delta_2>0$ is sufficiently small.
Thus, the uniform SLLN is proven.
\end{proof}
Similarly, 
the continuity of $\Phi(\cdot,\;P)$ on $\Theta_k^{\ast}(M)$ can be proven.
\begin{prop}\label{prop:C}
Let $M>0$ be an arbitrary number. 
Suppose that $\int\phi(\|\bm{x}\|)P(d\bm{x})$.
Then, 
$\Phi(\cdot,\;P)$ is continuous on $\Theta_k^{\ast}(M)$.
\end{prop}
\begin{proof}
If $(F,\;A),\;(G,\;B)\in \Theta_k^{\ast}$ are select such that $d_{H}(F,\;G)<\delta_1$ and $\|A-B\|_F<\delta_2$, 
then for each $\bm{g}\in G$, 
there exists $\bm{g}(\bm{f})\in F$ with $\|\bm{g}-\bm{g}(\bm{f})\|<\delta_1$, 
and furthermore,
\begin{align}\label{eq:4.2}
&\Phi(F,\;\bm{A},\;P)-\Phi(G,\;\bm{B},\;P) \nonumber \\
&=
\int\left[\min_{\bm{f}\in F}\phi(\|\bm{x}-\bm{Af}\|)-\min_{\bm{g}\in G}\phi(\|\bm{x}-\bm{Bg}\|)\right]P(d\bm{x})\nonumber \\
&\le 
\int \max_{\bm{g} \in G}[\phi(\|\bm{x}-\bm{Af}(\bm{g})\|)-\phi(\|\bm{x}-\bm{Bg}\|)]P(d\bm{x})\nonumber \\
&\le
\int \sum_{\bm{g} \in G}[\phi(\|\bm{x}-\bm{Bg}\|+M\delta_2+\delta_1)
-\phi(\|\bm{x}-\bm{Bg}\|)]P(d\bm{x})
\nonumber \\
&\le
k\sup_{\|\bm{x}\|\le R}\max_{\bm{g}\in G}[\phi(\|\bm{x}-\bm{Bg}\|+M\delta_2+\delta_1)
-\phi(\|\bm{x}-\bm{Bg}\|)]\nonumber \\
&\quad+
2k\lambda\int_{\|\bm{x}\|\ge R} \phi(\|\bm{x}\|)P(d\bm{x})
\end{align}
for $R>\delta_1+M(1+\delta_2)$.
When a sufficiently large $R$ and a sufficiently small $\delta_1,\;\delta_2>0$ are selected,
the last bound is less than $\epsilon$.
For each $\bm{f}\in F$, there also exists $\bm{f}(\bm{g})\in G$ with $\|\bm{f}-\bm{f}(\bm{g})\|<\delta_1$.
Therefore, 
the other inequality necessary for the continuity is obtained by interchanging 
$(F,\;A)$ and $(G,\;B)$ in the inequality $(\ref{eq:4.2})$.
\end{proof}
\section{The consistency theorem}\label{section:5}
\subsection{The existence of the population global optimizers}
The aim of this paper is to prove that, for a fixed measure $P$ satisfying some natural assumptions, 
the infimum distance between the (measurable) estimator $\hat{\theta}_n$ with $\Phi(\hat{\theta}_n)=m_k(P_n)$ 
and parameters achieving $m_k(P)$ converges almost surely to $0$, as the sample size goes to infinity. 
However, there may be no such parameters.
Thus, before providing the consistency theorem, 
the sufficient condition for the existence of parameters achieving $m_k(P)$ in $\Xi_k$ is provided.
The following proposition ensures the existence of such parameters.
The proof and some details about the proposition are given in Appendix \ref{app}.
\begin{prop}\label{prop:A}
Suppose that $\int\phi(\|\bm{x}\|)P(d\bm{x})<\infty$ and that $m_j(P) > m_{k}(P)$
for $j=1,\;2,\;\dots,\;k-1$.
Then, $\Theta^{\prime}\neq \emptyset$.
\end{prop}
From Lemma $\ref{lemma:B}$ in Appendix \ref{app}, 
there exists $M>0$ such that 
$F \subset B_{q}(5M)$  for all $(F,\;A)\in \Theta^{\prime}$.
Moreover, 
under the assumption of Proposition $\ref{prop:A}$, 
the following identification condition can be proven:
$$
\inf_{\theta\in \Theta_k^{\ast}(5M):d(\theta,\;\Theta^{\prime})\ge \epsilon}\Phi(\theta,\;P)
>
\inf_{\theta\in \Theta^{\prime}}\Phi(\theta,\;P)
\quad 
\text{for all $\epsilon >0$.}
$$
The proof of the identification condition is also given in Appendix \ref{app}.
The identification condition is used in Section $\ref{section:6}$.
\subsection{Strong consistency of reduced $k$-means clusterings}
If the parameter space is $\Theta_k^{\ast}(M)$, 
the strong consistency of RKM clustering can be proven.
Note that since $\Theta_k^{\ast}(M)$ is compact, 
we have $\Theta^{\ast}\neq \emptyset$ and the identification condition:
$$
\inf_{\theta \in \Theta_{\epsilon}^{\ast}(M)}\Phi(\theta,\;P) 
> \inf_{\theta\in \Theta^{\ast}}\Phi(\theta,\;P) 
\quad 
\text{for all }
\epsilon >0,
$$
where $\Theta_{\epsilon}^{\ast}(M):=\{\theta \in \Theta_k^{\ast}(M) \mid d(\theta,\;\Theta^{\ast})\ge \epsilon\}$.
\begin{prop}\label{prop:SC}
Suppose that $\int \phi(\|\bm{x}\|)P(d\bm{x})<\infty$. 
Then, for each $M>0$,
$$
\lim_{n\rightarrow \infty}d(\hat{\theta}^{\ast}_n,\;\Theta^{\ast})=0\quad
\mathrm{a.s.}
,\;\text{and }
\lim_{n\rightarrow \infty} m_k^{\ast}(P_n\mid M)=m_k^{\ast}(P\mid M)\quad
\mathrm{a.s.}
$$
\end{prop}
\begin{proof}
Since the uniform SLLN and the continuity of $\Phi(\cdot,\;P)$,
the proof of this proposition is given by the similar argument of the proof of the following consistency theorem.
\end{proof}

In a study by Pollard (1981), 
the uniqueness of the parameter is also assumed for the strong consistency theorem. 
As discussed in Section $\ref{section:3}$,
we cannot assume the uniqueness condition.
Thus, 
the condition that $m_j(P) > m_{k}(P)$ for $j=1,\;2,\;\dots,\;k-1$ 
is assumed instead of the uniqueness condition.

This condition is equivalent to 
the distinctness condition that 
$F(k)$ has $k$ distinct points for all $(F(k),\;A(k))\in \Theta^{\prime}(k)$. 
Indeed, 
suppose that there exists $\theta=(F(k),\;A(k))\in \Theta^{\prime}(k)$ such that
$F(k)$ have $k-1$ or fewer distinct points; that is, $\#(F(k))<k$.
There exists $i\in \mathbb{N}$ such that $i<k$ and $\theta \in \Xi_{k}$.
Then, $m_i(P)=m_k(P)$, which contradicts to $m_i(P)>m_k(P)$.
Thus, 
the condition that $m_j(P) > m_{k}(P)$ for $j=1,\;2,\;\dots,\;k-1$ implies
the distinctness condition.
Moreover, 
this condition is equivalent to $m_{k-1}(P)>m_{k}(P)$ since $m_{k}(P) \ge m_l(P)$ for each $k,\;l\in\mathbb{N}$ satisfying $k<l$.

The following main theorem gives the sufficient condition for the strong consistency of the estimator of RKM clustering.
\begin{theorem}\label{theorem:1}
Suppose that $\int \phi(\|\bm{x}\|)P(d\bm{x})<\infty$
and 
that $m_j(P) > m_{k}(P)$ for $j=1,\;2,\;\dots,\;k-1$.
Then, $\Theta^{\prime}\neq \emptyset$,
$$
\lim_{n\rightarrow \infty}d(\hat{\theta}_n,\;\Theta^{\prime})=0\quad
\mathrm{a.s.}
,\;\text{and }
\lim_{n\rightarrow \infty} m_k(P_n)=m_k(P)\quad
\mathrm{a.s.}
$$
\end{theorem}
\section{Proof of Theorem $\ref{theorem:1}$}\label{section:6}
Because almost sure convergence is dealt with, 
null sets of elements exists for which the convergence does not hold.
Hereafter, $\Omega_1$ denotes the set obtained by avoiding a proper null set from $\Omega$.
In the first step of the proof,
when $n$ is sufficiently large, 
the estimators of the cluster centers are contained within a compact ball that does not depend on $\omega \in \Omega$.
For convenience, it is assumed that $\phi(r) \rightarrow \infty$ as $r\rightarrow \infty$.
When $\phi$ is bounded, the proof is a little complicated. 

First, 
we prove the following lemma.
\begin{lemma}\label{lemma:1}
Suppose that $\int \phi(\|\bm{x}\|)P(d\bm{x})<\infty$.
Then, there exists $M>0$ such that
$$
P\left(\bigcup_{n=1}^{\infty}\bigcap_{m=n}^{\infty}\{\omega \mid \forall (F_{m},\;A_m) \in \Theta_m^{\prime};\; F_{m} (\omega)\cap B_q(M)\neq \emptyset\}\right)
=1.
$$
\end{lemma}
\begin{proof}
Select an appropriate value $r>0$ to satisfy the condition that 
the ball $B_p(r)$ has positive $P$ measure, i.e., 
$P(B_p(r))>0$.
Let $M$ be sufficiently large for satisfying $M>r$ and
\begin{align}\label{eq:3.1}
\phi(M-r)P(B_p(r))>\int \phi(\|\bm{x}\|)P(d\bm{x}).
\end{align}
From the definition of $\hat{\theta}_n=(F_n,\;A_n)$, 
$\Phi(F_n,\;A_n,\;P)\le \Phi(F_0,\;A,\;P)$ 
for any set $F_0$ containing at most $k$ points and any $A \in \mathcal{O}(p\times q)$.
The parameter $F_0$ is chosen such that it only consists of the origin.
Then, by SLLN, 
$$
\Phi(F_0,\;A,\;P_n) =\int \phi(\|\bm{x}\|)P_n(d\bm{x})
\rightarrow 
\int \phi(\|\bm{x}\|)P(d\bm{x})\quad \mathrm{a.s.},
$$
for each $A \in \mathcal{O}(p\times q)$.

Let $\Omega^{\prime}:= \{ \omega \in \Omega_1 \mid \forall n\in \mathbb{N};\;\exists m\ge n;\;\exists (F_{m},\;A_m) \in \Theta_m^{\prime};\;F_{m}(\omega)\cap B_q(M) =\emptyset \}$.
By the axiom of choice, for an arbitrary $\omega \in \Omega^{\prime}$ 
there exists a subsequence $\{n_l\}_{l\in \mathbb{N}}$ 
such that $n_s <n_t\;(s<t)$ and $F_{n_{l}}\cap B_q(M)=\emptyset$.
Thus, 
\begin{align*}
\lim\sup_{l}\Phi(F_{n_l},\;A_{n_l},\;P_{n_l}) 
&\ge
\lim\sup_{l} \frac{1}{n_l}\sum_{i\in \{i\mid \bm{X}_i\in B_p(r) \}}\min_{1\le j\le k}
\phi(\| \bm{X}_i -A_{n_l}\bm{f}_{j} \|)\\
&\ge  
\lim\sup_{l} \frac{1}{n_l}\sum_{i\in \{i\mid \bm{X}_i\in B_p(r) \}}\phi(M-r)\\
&=
\phi(M-r)\lim\sup_{l}P_{n_l}(B_p(r)) =\phi(M-r)P(B_p(r)).
\end{align*}
On the other hand, 
$\lim\sup_{l}m_{k}(P_{n_{l}}) \le \lim_{l}\Phi(F_0,\;A,\;P_{n_l})$
because $m_{k}(P_{n_{l}})\le \Phi(F_0,\;A,\;P_{n_l})$.
Therefore, we have
$
\lim\sup_{l} m_k(P_{n_l}) \le  \int \phi(\|\bm{x}\|)P(d\bm{x})
$
and
$
\lim\sup_{l}\Phi(F_{n_l},\;A_{n_l},\;P_{n_l}) > \int \phi(\|\bm{x}\|)P(d\bm{x})
$, 
which is a contradiction.
Therefore, $P(\Omega^{\prime})=0$, that is,
$$
P\left(\bigcup_{n=1}^{\infty}\bigcap_{m=n}^{\infty}\{\omega \mid \forall (F_{m},\;A_m) \in \Theta_m^{\prime};\; F_{m} (\omega)\cap B_q(M)\neq \emptyset\}\right)
=1.
$$
\end{proof}

Without loss of generality,
all $F_n$ can be assumed contain at least one point of $B_q(M)$ when $n$ is sufficiently large.
The next lemma shows that 
for sufficiently large $n$, there exists $M>0$ such that
the closed ball $B_q(5M)$ 
contains all estimators of centers.
When $k=1$, the next lemma is obviously satisfied.

From the results in Section $\ref{section:4}$ and using the same arguments in the final part of this section, 
the conclusions of the theorem are proven when $k=1$.
\begin{lemma}\label{lemma:2}
Under the assumption of the theorem, 
there exists $M>0$ such that
$$
P\left(\bigcup_{n=1}^{\infty}\bigcap_{m=n}^{\infty}\{\omega \mid \forall (F_{m},\,A_m) \in \Theta_m^{\prime};\;F_{m}(\omega)\subset B_q(5M)\}\right)
=1.
$$
\end{lemma}
\begin{proof}
Choose $M>0$ sufficiently large to satisfy the inequality $(\ref{eq:3.1})$ and
\begin{align}\label{eq:3.2}
\lambda\int_{\|\bm{x}\|\ge 2M}\phi(\|\bm{x}\|)P(d\bm{x})<\epsilon,
\end{align}
where $\epsilon >0$ is selected to ensure $\epsilon + m_k(P) <m_{k-1}(P)$.
Note that $m_{j}(P) \le m_{j}^\ast(P\mid M)$ for $j\in \mathbb{N}$.

Suppose that $F_n$ contains at least one center outside $B_q(5M)$ and consider the effect on $\Phi(F_n,\;A,\;P_n)$
by deleting such outside centers from $F_n$ for all $A \in \mathcal{O}(p\times q)$.
From Lemma $\ref{lemma:1}$, all $F_n$ contain at least one center on $B_q(M)$ when $n$ is sufficiently large, say $\bm{f}_1$.
In the worst case, 
the cluster of $\bm{f}_1 \in B_q(M)$ should contain all sample points belonging to clusters outside $B_q(5M)$.
Because these points must be outside $B(2M)$,  
the increment of $\Phi(F_n,\;A,\;P_n)$ due to the deletion of centers outside $B_q(5M)$ from $F_n$ 
would be at most
\begin{align*}
\int_{\|\bm{x}\|\ge 2M}\phi(\|\bm{x}-A\bm{f}_1\|)P_n(d\bm{x})
&\le
\int_{\|\bm{x}\|\ge 2M}\phi(\|\bm{x}\|+\|\bm{f}_1\|)P_n(d\bm{x})\\
&\le
\int_{\|\bm{x}\|\ge 2M}\phi(2\|\bm{x}\|)P_n(d\bm{x})\\
&\le
\lambda\int_{\|\bm{x}\|\ge 2M}\phi(\|\bm{x}\|)P_n(d\bm{x}).
\end{align*}
Denote the set obtained by deleting centers outside $B_q(5M)$ from $F_n$ by $F_n^{\ast}$.
For each $A\in \mathcal{O}(p\times q)$, 
$(F_n^{\ast},\;A)$ is contained in $\Theta_{k-1}^{\ast}(5M)$,
and thus,
$$
\Phi(F_{n}^{\ast},\;A,\;P_n)\ge m_{k-1}^{\ast}(P_n \mid 5M)\ge m_{k-1}(P_n).
$$
For each $\bm{x}$ satisfying $\|\bm{x}\|<2M$ and each $A\in \mathcal{O}(p\times q)$, 
we have
$$
\|\bm{x}-A\bm{f}\|>3M \quad \text{for all $\bm{f}\not\in B_q(5M)$}
$$
and 
$$
\|\bm{x}-A\bm{g}\|<3M \quad \text{for all $\bm{g}\in B_q(M)$}.
$$
Thus, 
$$
\int_{\|\bm{x}\| < 2M} \min_{\bm{f}\in F_{n}}\phi(\|\bm{x}-A\bm{f}\|)P_n(d\bm{x})
=
\int_{\|\bm{x}\| < 2M} \min_{\bm{f}\in F_{n}^{\ast}}\phi(\|\bm{x}-A\bm{f}\|)P_n(d\bm{x}).
$$
for all $A\in \mathcal{O}(p\times q)$.
Note that
$$
\lim_{n \rightarrow \infty}m_{k-1}^{\ast}(P_n\mid 5M) = m_{k-1}^\ast(P\mid 5M)\quad\mathrm{a.s.}
$$
by Proposition $\ref{prop:SC}$.

Let $\Omega^{\ast}:=\left\{\omega \in \Omega_1\mid \forall n \in \mathbb{N};\,\exists m \ge n;\,
\exists (F_{m},\,A_m) \in \Theta_m^{\prime};\,F_m(\omega)\not\subset B_q(5M) \right\}$.
By the axiom of choice, for an arbitrary $\omega \in \Omega^{\ast}$
there exists a subsequence $\{n_l\}_{l\in \mathbb{N}}$ such that $n_s <n_t\;(s<t)$ and $F_{n_l}(\omega)\not\subset B_q(5M)$.
For any $F$ with $k$ or fewer points and any $A\in \mathcal{O}(p\times q)$,
\begin{align}\label{eq:3.3}
 m_{k-1}^\ast(P\mid 5M)
&\le 
\lim\inf_i\Phi(F_{n_{i}}^{\ast},\;A_{n_{i}},\;P_{n_{i}}) 
\le 
\lim\sup_i\Phi(F_{n_{i}}^{\ast},\;A_{n_{i}},\;P_{n_{i}}) \nonumber\\
&=
\lim\sup_i
\biggl[
\int_{\|\bm{x}\| < 2M} \min_{\bm{f}\in F_{n_{i}}}\phi(\|\bm{x}-A_{n_i}\bm{f}\|)P_{n_{i}}(d\bm{x})\nonumber\\
&\qquad\qquad\quad+
\int_{\|\bm{x}\| \ge 2M} \min_{\bm{f}\in F_{n_{i}}^{\ast}}\phi(\|\bm{x}-A_{n_i}\bm{f}\|)P_{n_{i}}(d\bm{x})
 \biggr] \nonumber\\
&\le
\lim\sup_{n}\left[ \Phi(F_n,\;A_n,\;P_n)+\lambda\int_{\|\bm{x}\|\ge 2M}\phi(\|\bm{x}\|)P_n(d\bm{x})\right] \nonumber\\ 
&\le
\lim\sup_{n}\Phi(F,\;A,\;P_n)+\lambda\int_{\|\bm{x}\|\ge 2M}\phi(\|\bm{x}\|)P(d\bm{x}).
\end{align}
Set $(F,\;A) \in \Theta^{\prime}$; that is, $m_k(P)=\Phi(F,\;A,\;P)$.
From the requirement of $M>0$ in the inequality $(\ref{eq:3.2})$ and SLLN, 
the last bound of the inequality $(\ref{eq:3.3})$ is less than
$$
\Phi(F,\;A,\;P)+\epsilon =m_k(P)+\epsilon<m_{k-1}(P).
$$
This is a contradiction.
Thus, the following is obtained
$$
P\left(\bigcup_{n=1}^{\infty}\bigcap_{m=n}^{\infty}\{\omega \mid \forall (F_{m},\,A_m) \in \Theta_m^{\prime};\;F_{m}(\omega)\subset B_q(5M)\}\right)
=1.
$$
\end{proof}

For sufficiently large $n$, all $F_n$ values satisfying 
$$
\inf_{A\in \mathcal{O}(p\times q)}\Phi(F_{n},\;A,\;P_n)=m_{k}(P_n)
$$
lie in $\mathcal{R}_{k}^\ast(5M)$.
From Proposition $\ref{prop:A}$ and Lemma $\ref{lemma:B}$, $\mathcal{R}_{k}^\ast(5M)$ contains all optimal sets 
satisfying 
$$
\inf_{A\in \mathcal{O}(p\times q)}\Phi(F,\;A,\;P)=m_{k}(P).
$$
It also follows that
Pollard (1981) assume that it is large enough to satisfy that $\mathcal{R}_{k}^\ast(5M)$ contains the optimal cluster centers, as the requirement on $M$, 
but this requirement is also unnecessary.

In a similar way of Theorem $5.14$ (van der Vaart, 1998), 
if we obtain the continuity of $\Phi(\cdot,\;\cdot,\;P)$ and the uniform SLLN, i.e.,
$$
\sup_{(F,\;A)\in \Theta_{k}^{\ast}(5M)}| \Phi(F,\;A,\;P_n)-\Phi(F,\;A,\;P) | \stackrel{\rm a.s.}{\longrightarrow} 0,
$$
the theorem is completely proven.

Let 
$$
\tilde{\theta}_n
=
\begin{cases}
\hat{\theta}_{n} & \text{if } \hat{\theta}_{n}\in \Theta_{k}^{\ast}(5M) \\
\theta_{\ast} & \text{if } \hat{\theta}_{n}\not\in \Theta_{k}^{\ast}(5M)
\end{cases},
$$
where $\theta_{\ast}\in \Theta_{k}^{\ast}(5M)$ is chosen to ensure $d(\theta_{\ast},\;\Theta^{\prime})>0$.
Then, 
for a sufficiently large $n$, $\tilde{\theta}_n=\hat{\theta}_n$ by Lemma $\ref{lemma:2}$, 
and 
the following condition is obtained
$$
\lim\sup_n\left[\Phi(\tilde{\theta}_n,\;P_n)- \inf_{\theta\in \Theta^{\prime}}\Phi(\theta,\;P_n)\right]\le 0\quad\mathrm{a.s.}
$$
Since $\lim\sup_{n}\Phi(\theta_0,\;P_n) =\Phi(\theta_0,\;P)\;(=m_k(P))$ for any fixed $\theta_0 \in \Theta^{\prime}$,
$$
\lim\sup_{n}\inf_{\theta\in \Theta^{\prime}}\Phi(\theta,\;P_n)\le \lim\sup_{n}\Phi(\theta_0,\;P_n) = m_k(P)\quad\mathrm{a.s.}
$$
Thus, 
\begin{align}\label{eq:3.4}
0
&\ge
\lim\sup_n\Phi(\tilde{\theta}_n,\;P_n)-\lim\sup_n\inf_{\theta\in \Theta^{\prime}}\Phi(\theta,\;P_n)
\nonumber \\
&\ge
\lim\sup_n\Phi(\tilde{\theta}_n,\;P_n)-m_k(P)\quad
\mathrm{a.s.}
\end{align}
Let $\Theta_{\epsilon}^{\ast}(5M):=\{\theta \in \Theta_k^{\ast}(5M) \mid d(\theta,\;\Theta^{\prime})\ge \epsilon\}$ for each $\epsilon>0$.
From the uniform SLLN, 
\begin{align}\label{eq:3.5}
\lim\inf_{n}\inf_{\theta\in \Theta_{\epsilon}^{\ast}(5M)}\Phi(\theta,\;P_n) 
\ge
\inf_{\theta\in \Theta_{\epsilon}^{\ast}(5M)}\Phi(\theta,\;P)\quad \mathrm{a.s.}
\end{align}
for all $\epsilon >0$.
An arbitrary $\epsilon>0$ is selected.
From Corollary $\ref{col:A}$ and the inequalities $(\ref{eq:3.4})$ and $(\ref{eq:3.5})$,
we have
\begin{align}\label{eq:main}
\lim\inf_{n}\inf_{\theta\in \Theta_{\epsilon}^{\ast}(5M)}\Phi(\theta,\;P_n) 
>
\lim\sup_n\Phi(\tilde{\theta}_n,\;P_n)
\quad
\mathrm{a.s.}
\end{align}
That is, for any $\omega \in \Omega$ satisfying the inequality $(\ref{eq:main})$,
there exists $n_0\in \mathbb{N}$ such that 
$$
\inf_{\theta\in \Theta_{\epsilon}^{\ast}(5M)}\Phi(\theta,\;P_n)
>
\Phi(\hat{\theta}_n,\;P_n)
=
\Phi(\tilde{\theta}_n,\;P_n)
$$
for all $n\ge n_0$.
Conversely, suppose that there exists $n\ge n_0$ such that $d(\hat{\theta}_n,\;\Theta^{\prime})\ge \epsilon$.
Then, we obtain
$$
\inf_{\theta\in \Theta_{\epsilon}^{\ast}(5M)}\Phi(\theta,\;P_n) = \Phi(\hat{\theta}_n,\;P_n),
$$
which is a contradiction.
Thus, we obtain that $d(\hat{\theta}_n,\;\Theta^{\prime})< \epsilon$ for all $n\ge n_0$.
That is,
\begin{align*}
\lim_{n\rightarrow \infty}d(\hat{\theta}_n,\;\Theta^{\prime})=0\quad
\mathrm{a.s.}
\end{align*}
is proven.
From the continuity of $\Phi(\cdot,\;P)$,
the following is obtained:
\begin{align*}
\lim_{n\rightarrow \infty} m_k(P_n)=m_k(P)\quad
\mathrm{a.s.}
\end{align*}
\section{Selection of the number of dimensions}\label{section:7}

In RKM clustering,
the numbers of clusters and dimensions, $k$ and $q$, have to be appropriately
determined such that the cluster result can be optimized.
For determining the number of cluster, 
Wang (2010) proposed a new selection criterion based on clustering stability.
This criterion can be applied for determining other turning parameters with some clustering method
(e.g., Sun et al., 2012).

In this section, 
we propose a new simple criterion for determining the number of dimensions under given cluster number, 
which is not based on clustering stability.
We also propose a consistent estimator of the criterion.
Moreover,
we illustrate the effectiveness of the criterion through numerical experiments.
\subsection{New criterion for determining the number of dimensions}
First,
we define a variance ratio criterion for a population distribution $P$ by
$$
VR(q\mid P) 
:=
\inf_{(F,\;A)\in \Theta^{\prime}}
\frac{\int \min_{\bm{f}\in F}\|A^T\bm{x}-\bm{f}\|^2P(d\bm{x})}{\int \|A^T\bm{x}-A^T\bm{\mu}\|^2P(d\bm{x})},
$$
where 
$\bm{\mu} = \int \bm{x} P(d\bm{x})$.

Here, we assume that 
the population global optimal coefficient matrices are determined uniquely without the rotational indeterminacy of $A$, 
that is, 
there exists $(F_0,\;A_0)\in \Theta^{\prime}$ such that for all $(F,\;A)\in \Theta^{\prime}$ 
there exists $R\in \mathcal{O}(q)$ such that $A_0=AR$.
Let $(F,\;A),\;(F_\ast,\;A_\ast)\in \Theta^{\prime}$ with $F\neq F_\ast$ or $A\neq A_\ast$.
We have $\Phi(F,\;A,\;P)=\Phi(F_\ast,\;A_\ast,\;P)$ and $\int\|A^T\bm{x}\|^2P(d\bm{x})=\int\|A^T_\ast \bm{x}\|^2P(d\bm{x})$.
Since $$
\Phi(F,\;A,\;P)
=
\int \|\bm{x}\|^2P(d\bm{x}) -\int \|A^{T}\bm{x}\|^2P(d\bm{x})+
\int \min_{\bm{f}\in F}\|A^T\bm{x}-\bm{f}\|^2P(d\bm{x}),
$$
we obtain
$$
\frac{\int \min_{\bm{f}\in F}\|A^T\bm{x}-\bm{f}\|^2P(d\bm{x})}{\int \|A^T\bm{x}-A^T\bm{\mu}\|^2P(d\bm{x})}
=
\frac{\int \min_{\bm{f}_\ast\in F_\ast}\|A_\ast^T\bm{x}-\bm{f}_\ast\|^2P(d\bm{x})}{\int \|A_\ast^T\bm{x}-A_\ast^T\bm{\mu}\|^2P(d\bm{x})}.
$$

Unfortunately, 
we cannot obtain the value of this criterion 
since the population distribution is unknown.
However, 
we can construct a consistent estimator of $VR(q\mid P)$.
We define a estimator of $VR(q\mid P)$ by
$$
\widehat{VR}(q\mid P_n):=
\frac{\int \min_{\hat{\bm{f}}_n\in \hat{F}}\|\hat{A}_n^T\bm{x}-\hat{\bm{f}}_n\|^2P_n(d\bm{x})}
{\int \|\hat{A}_n^T\bm{x}-\hat{A}_n^T\bm{\mu}\|^2P_n(d\bm{x})}, 
$$
where $\hat{\theta}_n=(\hat{F}_n,\;\hat{A}_n)$.
The following theorem gives the sufficient conditions of the strong consistency of the estimator $\widehat{VR}(q\mid P_n)$.
\begin{theorem}\label{theorem:2}
Suppose that $\int \phi(\|\bm{x}\|)P(d\bm{x})<\infty$
and 
$m_1(P)>m_2(P)>\cdots>m_k(P)$.
Then,
$$
\int \|A^T\bm{x}-A^T\bm{\mu}\|^2P(d\bm{x})>0 \quad\text{for all } (F,\;A)\in \Theta^{\prime}
$$
and 
$$
\lim_{n\rightarrow \infty}\widehat{VR}(q\mid P_n)=VR(q \mid P)\quad
\mathrm{a.s.}
$$
\end{theorem}
\begin{proof}
Without loss of generality, 
we assume $\bm{\mu}=\bm{0}$.
First, we prove 
$$
\int \|A^T(\bm{x}-\bm{\mu})\|^2P(d\bm{x})>0 \quad\text{for all } (F,\;A)\in \Theta^{\prime}.
$$
Conversely, 
suppose that 
there exists $(F,\;A)\in \Theta^{\prime}(k)$ such that 
$\int \|A^{T}\bm{x}\|^2 P(d\bm{x})=0$.
Then, $\|A^{T}\bm{x}\|^{2}=0$ for all $\bm{x}$ in the support of $P$.
Since 
$$
\Phi(F,\;A,\;P)
=
\int \|\bm{x}-AA^{T}\bm{x}\|^2P(d\bm{x}) +
\int \min_{\bm{f}\in F}\|A^T\bm{x}-\bm{f}\|^2P(d\bm{x}),
$$
$F$ must contain zero.
Let $F_0:=\{\bm{0}\}\in \mathcal{R}_1$
and then $m_k(P)=\Phi(F_0,\;A,\;P)\ge m_1(P)$.
This is a contradiction.

Next, we prove the consistency of $\widehat{VR}(q\mid P_n)$.
From Theorem $\ref{theorem:1}$, 
we have 
$$
\lim_{n\rightarrow \infty}d(\hat{\theta}_n,\;\Theta^{\prime})=0\quad
\mathrm{a.s.}
$$
In the similar way as the proof of the uniform SLLN $(\ref{eq:4.1})$, 
we obtain 
\begin{align}\label{eq:7.1}
\lim_{n\rightarrow \infty}\sup_{A\in \mathcal{O}(p\times q) }\left| \int \|A^T\bm{x}\|^2\,P_n(d\bm{x}) -\int \|A^T\bm{x}\|^2\,P(d\bm{x}) \right| =0 \quad \mathrm{a.s.} 
\end{align}
and 
\begin{align}\label{eq:7.2}
\lim_{n\rightarrow \infty}\sup_{(F,\;A)\in \Theta }\left| \int \min_{\bm{f}\in F}\|A^T\bm{x}-\bm{f}\|^2 \,P_n(d\bm{x}) -\int \min_{\bm{f}\in F}\|A^T\bm{x}-\bm{f}\|^2 \,P(d\bm{x}) \right| =0 \quad \mathrm{a.s.} 
\end{align}
Let $\hat{\theta}_n=(\hat{F}_n,\;\hat{A}_n)$ and $(F,\;A)\in \Theta^{\prime}$.
We have 
$$
\lim_{n\rightarrow \infty} \int \|\hat{A}_n^T\bm{x}\|^2\,P_n(d\bm{x})=\int \|A^T\bm{x}\|^2\,P(d\bm{x})\quad\mathrm{a.s.}
$$
and
$$
\lim_{n\rightarrow \infty}\int \min_{\hat{\bm{f}}_{n}\in \hat{F}_n}\|\hat{A}_n^T\bm{x}-\hat{\bm{f}}_n\|^2 \,P_n(d\bm{x})
=
\int \min_{\bm{f}\in F}\|A^T\bm{x}-\bm{f}\|^2 \,P(d\bm{x})\quad\mathrm{a.s.}
$$
Therefore, 
we obtain
$$
\lim_{n\rightarrow \infty}\widehat{VR}(q\mid P_n)=VR(q \mid P)\quad
\mathrm{a.s.}
$$
\end{proof}

If the number of dimensions is determined larger than the optimal one, 
the subspace of RKM may be influenced from noise variables which do not have cluster structure.
Let $q_\ast$ be the optimal number of dimensions.
Define $VR(0\mid P):=0$ and $VR(q \mid P):=VR(q-1\mid P)$ for $q=\min\{k-1,\;p\}$.
Forward difference at $q_\ast$, $\Delta_{+}(q):=VR(q_\ast+1 \mid P)-VR(q_\ast \mid P)$, 
may be quite larger than backward difference at $q_\ast$, $\Delta_{-}(q):=VR(q_\ast \mid P)-VR(q_\ast-1 \mid P)$.
That is, for the optimal number of dimensions $q_\ast$, 
second order central difference at $q_\ast$, $\Delta_2(q_\ast):=VR(q_\ast+1 \mid P)-2VR(q_\ast \mid P)-VR(q_\ast-1 \mid P)$, 
may be larger than second order central difference at $q\;(q\neq q_\ast)$.
For example, 
we may estimate the optimal number of dimensions by 
$$
\hat{q}:=\arg\max_q \widehat{\Delta_2}(q),
$$
where $\widehat{\Delta_2}(q):=\widehat{VR}(q+1 \mid P)-2\widehat{VR}(q \mid P)-\widehat{VR}(q-1 \mid P)$.
\subsection{Numerical experiments}

In this subsection, 
we examine the effectiveness of the criterion through numerical experiments.
Let $K$ be the number of clusters, 
$q$ be the number of dimensions of the low dimensional space, 
$p_1$ be the number of the informative variables, 
$p_2$ be the number of the correlated noise variables,
and $p_3$ be the number of the independent noise variables. 
Denote $O_{p\times q}$ be the $p\times q$ zero matrix.
The $p_1\times q$ column wise orthogonal matrix is generated randomly, say $A_\ast$.
$K$ cluster centers in low-dimensional space are independently 
generated from the $q$-dimensional uniform distribution on $[-15,\;15]^{q}$,
say $\bm{f}_k\;(k=1,\;\cdots,\;K)$.
Cluster indicators are independently generated from the multinomial distribution 
for $K$ trials with equal probabilities,
say $\bm{u}_{i}=(u_{i1},\;\dots,\;u_{iK})\;(i=1,\;\dots,\;n)$.
Set $A=[A_\ast^T,\;O_{q\times (p_2+p_3)}]^T,\;\Sigma_{p_2}=(\sigma_{ij})_{p_2\times p_2}$ with $\sigma_{ii}=1$ and $\sigma_{ij}=0.25\;(i\neq j)$, and 
$$
\Sigma_p =
\begin{bmatrix}
I_{p_1} 			& O_{p_1\times p_2} & O_{p_1 \times p_3} \\
O_{p_2 \times p_1}	& \Sigma_{p_2}		& O_{p_2 \times p_3} \\
O_{p_3 \times p_1}	& O_{p_2 \times p_3}		& I_{p_3}
\end{bmatrix}.
$$   

The simulated data of $n$ observations, $\bm{x}_i\in \mathbb{R}^p\;(i=1,\;\dots,\;n)$, 
are generated as 
$$
\bm{x}_i = \sum_{k=1}^{K}u_{ik}\left( A\bm{f}_k+\bm{\epsilon}_{ik} \right),
$$
where 
$\bm{\epsilon}_{ik}$ are generated from the $p$-dimensional normal distribution $N(\bm{0},\;\Sigma_p)$.
Let $X=[\bm{x}_i,\;\dots,\;\bm{x}_n]^T$ and $Z$ be the normalized data matrix with zero means and unit variances. 

Here, 
we set $K=8$, $n=400$, $q=2$ or $3$ and $p_1=p_2=p_3=5$ or $10$.
We make $1000$ data sets for each setting, respectively.
Figure $\ref{HCS}$ shows hidden cluster structure $XA$ of the one of data set with setting $n=400,\;q=2$, and $p_1=p_2=p_3=5$.
\begin{figure} 
\includegraphics[scale=0.6]{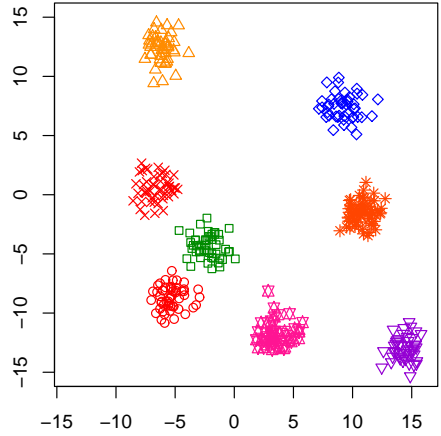}
\caption{Hidden cluster structure $XA$ of the one of data set with setting $n=400,\;q=2$, and $p_1=p_2=p_3=5$.}
\label{HCS}
\end{figure}
Figure $\ref{PCA}$ shows the first two principal components of PCA for $Z$, which is the same data set of Figure $\ref{HCS}$ 
and also shows that 
the first two principal components do not reflect the cluster structure. 
\begin{figure}
\includegraphics[scale=0.6]{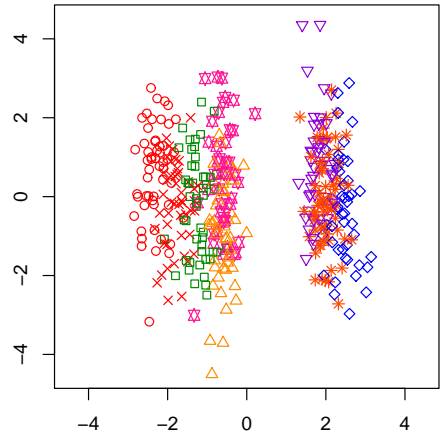}
\caption{First two dimensions of the principal scores of PCA for $Z$, which is the same data set of Figure $\ref{HCS}$
 (ARI of the tandem clustering with first two principal scores is $0.26$).}
\label{PCA}
\end{figure}
Moreover, Figure $\ref{RKM}$ shows the subspace of RKM with $q=2$ for $Z$, which is the same data set of Figure $\ref{HCS}$.
\begin{figure} 
\includegraphics[scale=0.6]{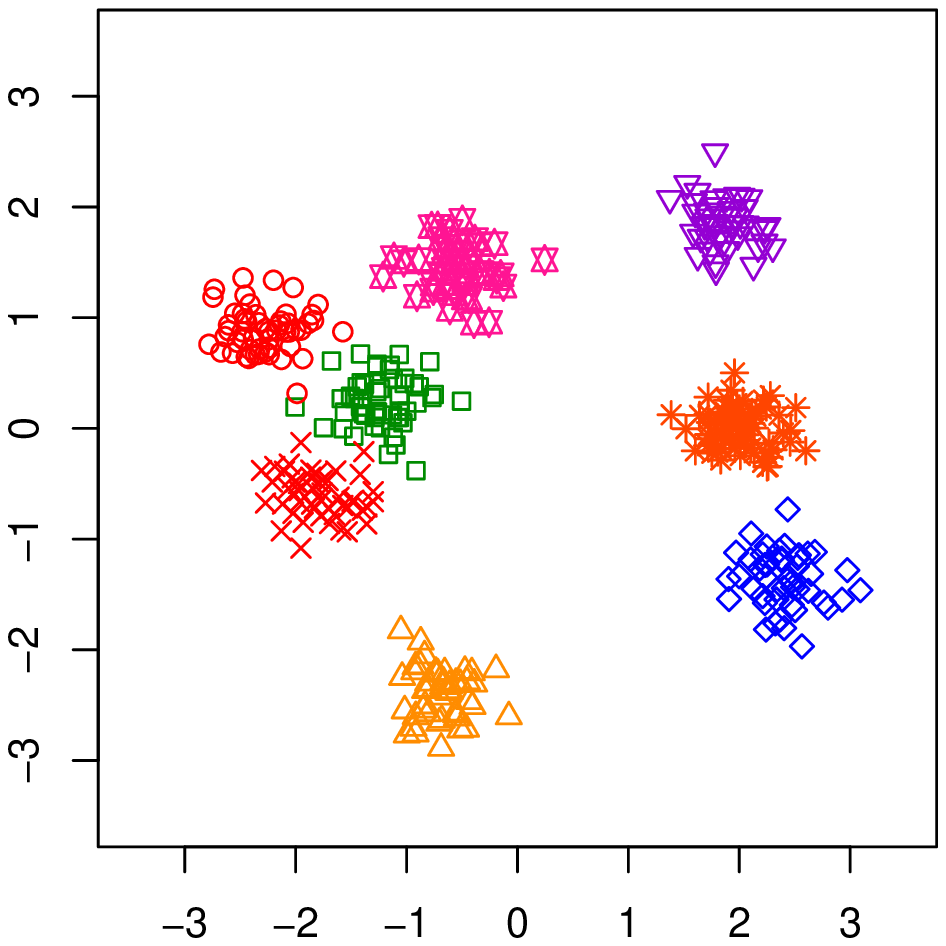}
\caption{The subspace of RKM for $Z$, which is the same data set of Figure $\ref{HCS}$
 (ARI of the RKM clustering with $q=2$ is $0.99$).}
\label{RKM}
\end{figure}
Figure $\ref{ARI}$ shows the adjusted rand indexes (ARI), which is proposed by Hubert and Arabie (1985), of RKM clustering with each number of dimensions of subspace.
In Figure $\ref{ARI}$, we can see that the number of dimensions of the subspace is quite important to the clustering result.  
\begin{figure} 
\includegraphics[scale=0.6]{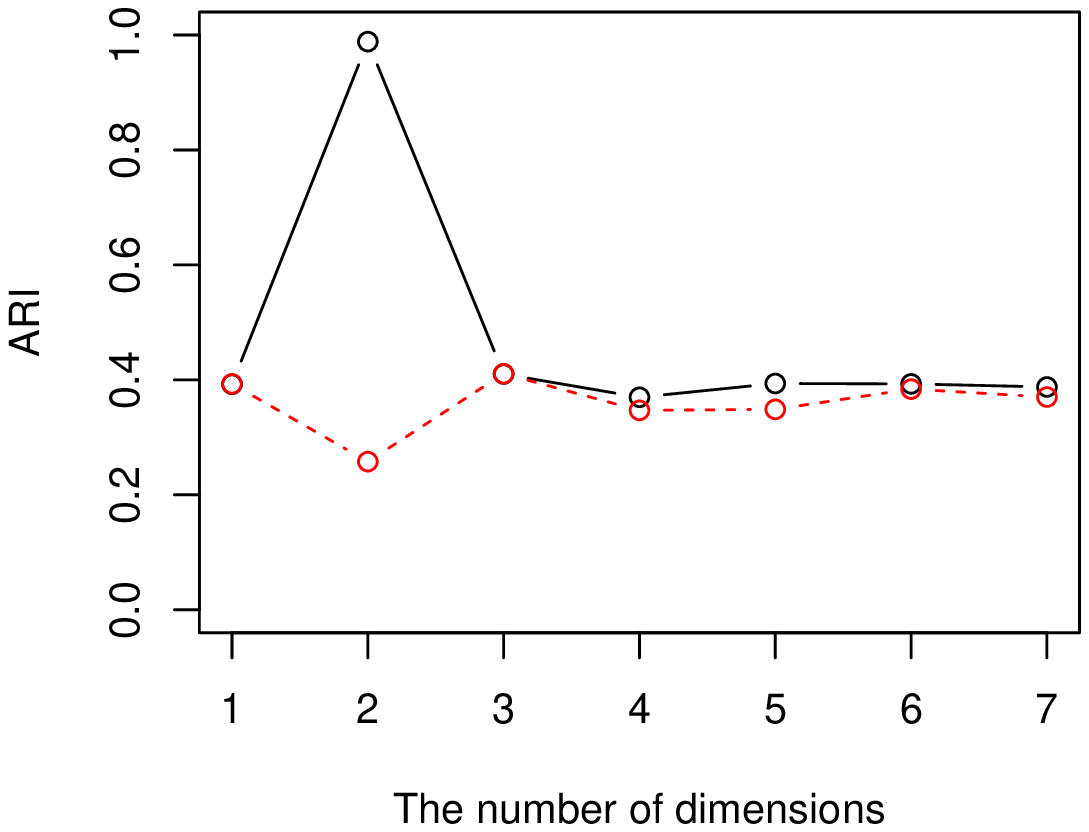}
\caption{ARI scores of RKM and tandem clustering with $q=1,\;2,\;\dots,\;7$ for $Z$, which is the same data set of Figure $\ref{HCS}$. Solid line is corresponded to ARI scores of RKM clustering and dash line is corresponded to ARI scores of tandem clustering.}
\label{ARI}
\end{figure}
Figure $\ref{VR}$ and $\ref{SD}$ show that $\widehat{VR}(q)$ and $\widehat{\Delta_2}(q)$ are useful for estimating the optimal number of dimensions. 
\begin{figure} 
\includegraphics[scale=0.6]{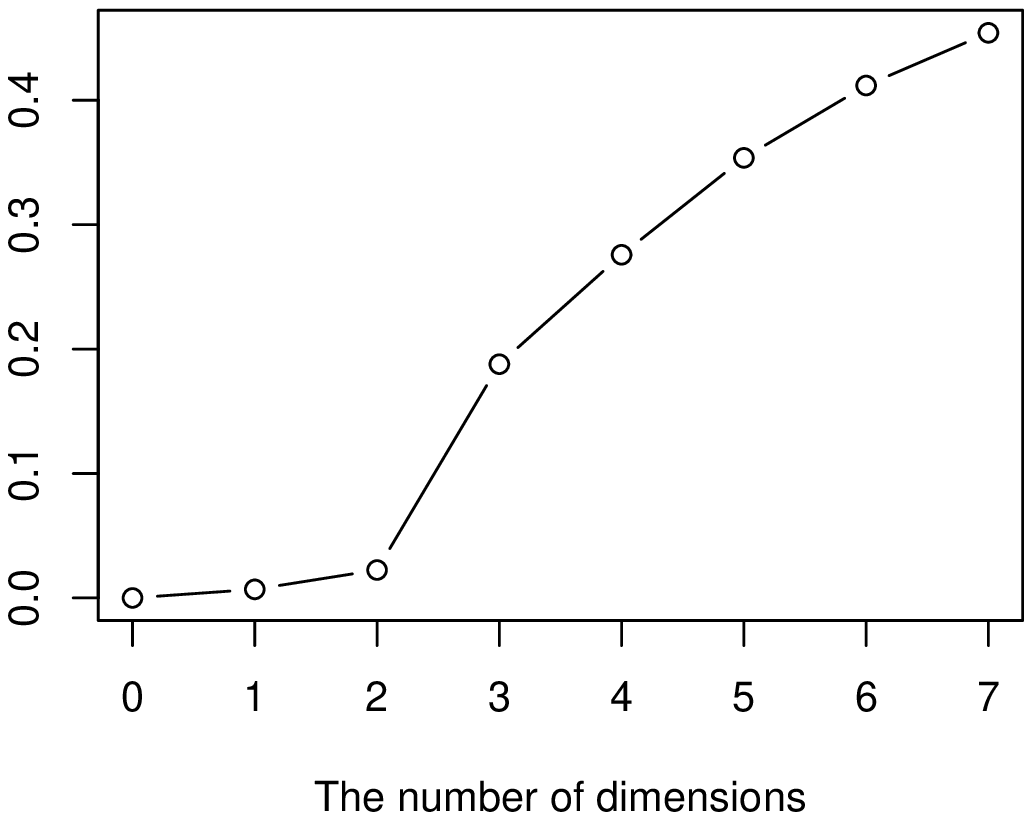}
\caption{$\widehat{VR}(q)$ scores of RKM with $q=1,\;2,\;\dots,\;7$ for $Z$, which is the same data set of Figure $\ref{HCS}$.}
\label{VR}
\end{figure}

\begin{figure} 
\includegraphics[scale=0.6]{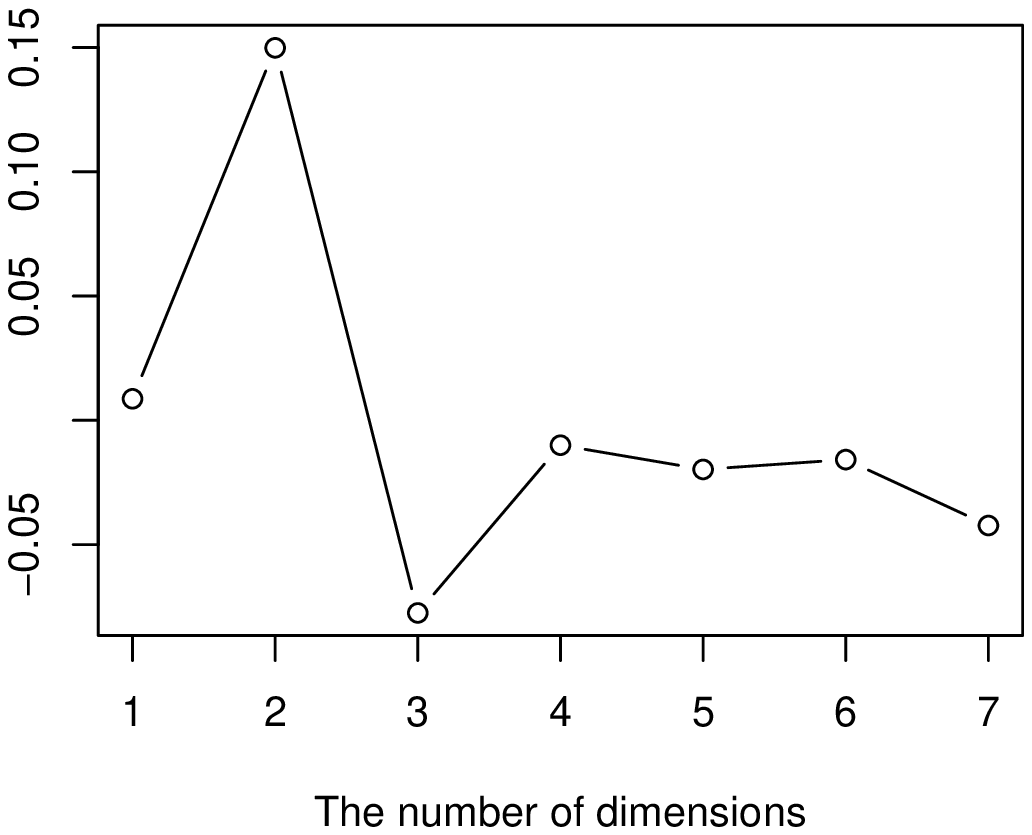}
\caption{$\widehat{\Delta_2}(q)$ scores of RKM with $q=1,\;2,\;\dots,\;7$ for $Z$, which is the same data set of Figure $\ref{HCS}$.}
\label{SD}
\end{figure}
Indeed, 
Table $\ref{AR}$ shows the agreement rates, of the choices by $\hat{q}$ and 
the optimal number $q_\ast:=\arg\max_q ARI(q)$, with each setting for $1000$ data sets.

\begin{table*}
\caption{Agreement rates with each setting for $1000$ data sets.}
\label{AR}
\begin{tabular}{crrrrc}
\hline
$q$ & $p_1=p_2=p_3$ & agreement rate \\
\hline
2 & 5	& 0.84 (837/1000)\\
2 & 10	& 0.95 (947/1000)\\
3 & 5	& 0.73 (726/1000)\\
3 & 10	& 0.89 (890/1000)\\
\hline
\end{tabular}
\end{table*}

\section{Conclusion}
This paper proves the strong consistency of RKM clusterings under i.i.d. sampling 
on the basis of the proof for the conventional $k$-means clustering provided by Pollard (1981).
Since our proof is based on the usual Blum-DeHardt uniform SLLN 
which requires only stationarity and ergodicity (e.g., Peskir, 2000), 
we can obtain the same results for a stationary ergodic process.

Under the i.i.d. condition, 
we can derive the rate of convergence for the convergence of the empirically optimal clustering scheme
if the support of the population distribution is bounded; that is, $P(\|\bm{X}_1\|^2\le B)=1$ for some $B>0$.
From Theorem $1$ in Linder et al. (1994), 
for all $\epsilon>0$ and $n(\epsilon/8B)^2\ge 2$ we can obviously obtain
\begin{align*}
P\left[ |m_k(P_n)-m_k(P)|  \ge \epsilon \right]
&\le
2P\left[ \sup_{\theta\in \Xi_k}|\Phi(\theta,\;P_n)-\Phi(\theta,\;P)|  \ge \epsilon \right]\\
&\le
2P\left[ \sup_{F \in \mathcal{R}_k^{(p)}}|KM(F,\;P_n)-KM(F,\;P)|  \ge \epsilon \right]\\
&\le
8(2n)^{k(p+1)}\exp\left( -\frac{n\epsilon^2}{512B^2}\right),
\end{align*}
where $\phi(r)=r^2$, $\mathcal{R}_k^{(p)}:=\{ R \subset \mathbb{R}^{p} \mid \#(R)\le k \}$, and $KM(F,\;P):=\int \min_{\bm{f}\in F}\|\bm{x}-\bm{f}\|^2P(d\bm{x})$.

Considering the relationship between the conventional $k$-means clustering and RKM clustering, 
the results presented in this paper are applicable to the conventional $k$-means clustering.
The related methods of RKM clustering include 
factorial $k$-means (FKM) clustering proposed by Vichi \& Kiers (2001).
In Terada (2013), 
the strong consistency of FKM clusterings under i.i.d. sampling (or for a stationary ergodic process) has been proven.
The form of sufficient conditions for the strong consistency of FKM clustering 
is similar to the case of RKM clusterings.
Moreover, 
the new simple criterion for determining the number of dimensions under given cluster number
and the consistent estimator of the criterion have been proposed.
Through numerical experiments, 
the effectiveness of the criterion has been illustrated.

Future studies in this regard
will examine the rate of convergence of estimators of RKM clustering and 
will propose the criterion required to determine the number of clusters.


\appendix
\section{The existence of $\Theta^{\prime}$}\label{app}
The existence of the minimum points of $\Phi(\cdot,\;P)$ are proven.
\begin{lemma}\label{lemma:A}
Suppose that $\int\phi(\|\bm{x}\|)P(d\bm{x})<\infty$.
There exists $M>0$ such that, for all $F^{\prime}\in\mathcal{R}_k$ satisfying $F^{\prime}\cap B_q(M)=\emptyset$,
$$
\inf_{A\in \mathcal{O}(p\times q)}\Phi(F^{\prime},\;A,\;P) > \inf_{\theta\in \Theta_{k}^{\ast}(M)}\Phi(\theta,\;P).
$$
\end{lemma}
\begin{proof}
Argue by contradiction, suppose that for any $M>0$
there exists $F^{\prime}\in\mathcal{R}_k$ such that $F^{\prime}\cap B_q(M)=\emptyset$ and
\begin{align}\label{eq:A.1}
\inf_{A\in \mathcal{O}(p\times q)}\Phi(F^{\prime},\;A,\;P) \le \inf_{\theta\in \Theta_{k}^{\ast}(M)}\Phi(\theta,\;P).
\end{align}
Select an $r>0$ such that 
the ball $B_p(r)$ has a positive $P$-measure, i.e.,
$P(B_p(r))>0$.
A sufficient large $M$ is selected such that $M>r$ and inequality $(\ref{eq:3.1})$ is satisfied.
From the inequality $(\ref{eq:A.1})$, 
\begin{align*}
\int \phi(\|\bm{x}\|)P(\bm{x}) 
&\ge \inf_{\theta\in \Theta_{k}^{\ast}(M)}\Phi(\theta,\;P)
\ge \inf_{A\in \mathcal{O}(p\times q)}\Phi(F^{\prime},\;A,\;P)\\
&\ge \phi(M-r)P(B_p(r)).
\end{align*}
This is a contradiction.
\end{proof}
\begin{lemma}\label{lemma:B}
Suppose that $\int\phi(\|\bm{x}\|)P(d\bm{x})<\infty$ and $m_j(P) > m_{k}(P)$ for $j=1,\;2,\;\dots,\;k-1$.
There exists $M>0$ such that, for any $F^{\prime}\in\mathcal{R}_k$ satisfying $F^{\prime}\not\subset B_q(5M)$,
$$
\inf_{A\in \mathcal{O}(p\times q)}\Phi(F^{\prime},\;A,\;P) 
> 
\inf_{\theta\in \Theta_{k}^{\ast}(5M)}\Phi(\theta,\;P).
$$
\end{lemma}
\begin{proof}
Select a sufficient large value $M>0$ to satisfy 
the inequalities $(\ref{eq:3.1})$ and $(\ref{eq:3.2})$.
To obtain a contradiction, suppose that for all $M>0$ there exists $F^{\prime}\in\mathcal{R}_k$ satisfying $F^{\prime}\not\subset B_q(5M)$ and
$$
\inf_{A\in \mathcal{O}(p\times q)}\Phi(F^{\prime},\;A,\;P) 
\le 
\inf_{\theta\in \Theta_{k}^{\ast}(5M)}\Phi(\theta,\;P).
$$
Let $\mathcal{R}_k^{\prime}$ be the set of such $F^{\prime}$ so that
$$
m_{k}(P)=\inf_{\theta\in \mathcal{R}_k^{\prime}\times \mathcal{O}(p\times q)}\Phi(\theta,\;P).
$$
From Lemma $\ref{lemma:A}$, 
each $F^{\prime}\in \mathcal{R}_k^{\prime}$ includes at least one element in $B_{q}(M)$, say $\bm{f}_1$.

For any $\bm{x}$ satisfying $\|\bm{x}\|<2M$ and any $A\in\mathcal{O}(p\times q)$,
$$
\|\bm{x}-A\bm{f}\|>3M \quad \text{for all $\bm{f}\not\in B_q(5M)$}
$$
and 
$$
\|\bm{x}-A\bm{g}\|<3M \quad \text{for all $\bm{g}\in B_q(M)$}.
$$
Let $F^{\ast}$ denote the set obtained by deleting all elements outside $B_q(5M)$ from $F^{\prime}$.
Then, 
$$
\int_{\|\bm{x}\| < 2M} \min_{\bm{f}\in F^{\prime}}\phi(\|\bm{x}-A\bm{f}\|)P(d\bm{x})
=
\int_{\|\bm{x}\| < 2M} \min_{\bm{f}\in F^{\ast}}\phi(\|\bm{x}-A\bm{f}\|)P(d\bm{x}).
$$
Since
$
\int_{\|\bm{x}\| \ge 2M}\phi(\|\bm{x}-A\bm{f}_1\|)P(d\bm{x}) 
\le
\lambda \int_{\|\bm{x}\| \ge 2M}\phi(\|\bm{x}\|)P(d\bm{x})
$, we obtain
\begin{align*}
&\Phi(F^{\prime},\;A,\;P) + \lambda \int_{\|\bm{x}\| \ge 2M}\phi(\|\bm{x}\|)P(d\bm{x})\\
&\ge
\int_{\|\bm{x}\| < 2M} \min_{\bm{f}\in F^{\ast}}\phi(\|\bm{x}-A\bm{f}\|)P(d\bm{x})
+
\int_{\|\bm{x}\| \ge 2M}\phi(\|\bm{x}-A\bm{f}_1\|)P(d\bm{x})\\
&\ge
\Phi(F^{\ast},\;A,\;P) \ge m_{k-1}(P)
\end{align*}
for all $A\in \mathcal{O}(p\times q)$.
Therefore, 
we obtain
$$
m_k(P)+\epsilon \ge m_{k-1}(P).
$$
This contradicts $m_k(P)+\epsilon <m_{k-1}(P)$.
\end{proof}

We will denote the essential parameter space by $\Theta_{k}$; that is, $\Theta_{k}:=\Theta_k^{\ast}(5M)$.
By Lemma $\ref{lemma:B}$, 
$$
\inf_{\theta\in \Xi_k}\Phi(\theta,\;P)=\inf_{\theta\in \Theta_k}\Phi(\theta,\;P)
$$
and 
there is no $\theta \in (\mathcal{R}_k\setminus \mathcal{R}_k^{\ast}(5M))\times \mathcal{O}(p\times q)$ satisfying $m_k(P)=\Phi(\theta,\;P)$.

\begin{proof}[Proof of Proposition $\ref{prop:A}$]
First, it is proven that 
there exists a sequence $\{\theta_n\}_{n\in \mathbb{N}}$ in $\Theta_k$ such that
$\Phi(\theta_n,\;P)\rightarrow m_k(P)$ as $n\rightarrow \infty$.
Let $C:=\{\Phi(\theta,\;P)\mid \theta \in \Theta_k\}$ and $m_k(P)=\inf C$.
For all $x >m_k(P)$, there exists $c<x$ in $C$.
Write $x_n:=m_k(P)+1/n$ and $C_n:=\{c\in C\mid c <x_n\}$.
Let $\mathfrak{P}(C)$ be the power set of $C$.
From the axiom of choice, 
there exists a function $f:\mathfrak{P}(C)\setminus\{\emptyset\}\rightarrow C$ such that 
$f(B)\in B$ for all $B\in \mathfrak{P}(C)\setminus\{\emptyset\}$.
Let $c_n:=f(C_n)$ and $x_n > c_n \ge m_k(P)$.
Thus, $c_n \rightarrow m_k(P)$ as $n\rightarrow \infty$.
Using the axiom of choice,
a sequence $\{\theta_n\}_{n\in\mathbb{N}}$ can be selected such that 
$\Phi(\theta_n,\;P)\rightarrow m_k(P)$ as $n\rightarrow \infty$.

From the compactness of $\Theta_k$, 
there exists a convergent subsequence of $\{\theta_n\}_{n\in\mathbb{N}}$, say $\{\theta_{m_i}\}_{i\in\mathbb{N}}$.
Let $\theta^{\ast}\in \Theta_k$ denote the limit of such subsequence, that is,
$\theta_{m_i}\rightarrow \theta^{\ast}$ as $i \rightarrow \infty$. 
Because $\Phi(\cdot,\;P)$ is continuous on $\Theta_k$, 
$\Phi(\theta^{\ast},\;P)=m_k(P)$.
That is, $\Theta^{\prime}\neq \emptyset$.
\end{proof}

The next corollary ensures the identification condition for $\Phi(\cdot,\;P)$.
\begin{col}\label{col:A}
Let $\Theta^{\prime}:=\{\theta_k\in \Theta_k\mid \Phi(\theta_k,\;P)=m_{k}(P)\}$.
Assume the assumptions of Lemma $\ref{lemma:B}$.
Then, 
$$
\inf_{\theta\in \Theta_k:d(\theta,\;\Theta^{\prime})\ge \epsilon}\Phi(\theta,\;P)
>
\inf_{\theta\in \Theta^{\prime}}\Phi(\theta,\;P)
\quad 
\text{for all $\epsilon >0$.}
$$
\end{col}
\begin{proof}
Let $\Theta_{\epsilon}:=\{\theta\in \Theta_k\mid d(\theta,\;\Theta^{\prime})\ge \epsilon\}$.
To obtain a contradiction, suppose that there exists $\epsilon > 0$ such that 
$
\inf_{\theta\in \Theta_{\epsilon}}\Phi(\theta,\;P)
=
\inf_{\theta\in \Theta^{\prime}}\Phi(\theta,\;P)
$.
Like in the proof of Proposition $\ref{prop:A}$,
there exists a sequence $\{\theta_n\}_{n\in\mathbb{N}}$ on $\Theta_{\epsilon}$ satisfying 
$\Phi(\theta_n,\;P)\rightarrow m_k(P)$ as $n\rightarrow \infty$.
From the compactness of $\Theta_k$, 
there exists a convergent subsequence of $\{\theta_n\}_{n\in\mathbb{N}}$, say $\{\theta_{m_i}\}_{i\in\mathbb{N}}$.
Let $\theta^{\ast}\in \Theta_k$ denote the limit of such subsequence and $\Phi(\theta^{\ast},\;P)=m_k(P)$,
that is, $\theta^{\ast}\in \Theta^{\prime}$.
On the other hand,
$d(\theta_{m_i},\;\theta^{\ast})< \epsilon$ for sufficiently large $i\in \mathbb{N}$ 
because $\theta_{m_i}\rightarrow \theta^{\ast}$ as $i \rightarrow \infty$.
Thus, $\theta_{m_i}\not\in \Theta_{\epsilon}$ for sufficiently large $i\in \mathbb{N}$.
This is a contradiction.
\end{proof}
\end{document}